\documentclass[12pt]{article}
 \usepackage{a4wide}

\usepackage{fancyhdr, amsmath, amssymb, amsfonts, url, mathrsfs, graphicx, epsfig,  amsthm}
\usepackage{tikz}
\usepackage{amscd}

\numberwithin{equation}{section}

\newtheorem{thm}{Theorem}[section]

\newtheorem{lemma}[thm]{Lemma}
\newtheorem{definition}[thm]{Definition}
\newtheorem{example}[thm]{Example}

\usepackage{pgfplots}

\AtEndDocument{\bigskip{\footnotesize
  \noindent
  \textsc{D. Parmenter, School of Mathematics, University of Bristol, Bristol, BS8 1UG, UK, and the Heilbronn Institute for Mathematical Research, Bristol, UK.}
  
  \noindent
  \textit{E-mail address}: \texttt{d.parmenter@bristol.ac.uk}
  \par
  \addvspace{\medskipamount}
  \noindent
  \textsc{M. Pollicott, Department of Mathematics, Warwick University, Coventry, CV4 7AL, UK}
  
  \noindent
  \textit{E-mail address}: \texttt{masdbl@warwick.ac.uk}
}}

\title{Constructing equilibrium  states for Smale spaces}
\author{David Parmenter and Mark Pollicott
\thanks{
The second author is partly
supported by ERC-Advanced Grant 833802-Resonances and EPSRC grant
EP/T001674/1}
}

\begin{document}

\maketitle

\abstract{There are several known constructions of equilibrium states for
H\"older continuous potentials in the context of
both subshifts of finite type and uniformly hyperbolic systems. 
In this article we present another method of building such measures, 
formulated in the unified and more general setting of Smale spaces.
This simultaneously extends the authors' previous work for hyperbolic attractors 
(modelled after Sinai's classical approach for SRB-measures) and gives a new and
original construction of equilibrium states for subshifts of finite type.}


\section{Introduction}

In this article we shall consider the classical problem on the relationship between equilibrium states for different potentials. Moreover for any two H\"older continuous potentials, we shall {\color{black} give} a geometric construction for transforming the Gibbs measure for one potential into the Gibbs measure for the other potential. The construction presents a new way to think about Gibbs measures complementing known constructions using, for example, periodic points \cite{bowen-periodic} or homoclinic points \cite{hr}.

We will work in the general setting of Smale spaces. Recall that a uniformly hyperbolic diffeomorphism has a local product structure by local stable and unstable manifolds (see \cite{bowen-SLN}).
A Smale space is an extension of the uniformly hyperbolic diffeomorphisms in the sense that we \textcolor{black}{only} have a compact metric space $X$ and a homeomorphism $f:X \rightarrow X$ satisfying a local product structure determined by \textcolor{black}{an appropriate} bracket map $[\cdot,\cdot] : X \times X \rightarrow X$.
Additionally, subshifts of finite type are a class of examples of Smale spaces
and our work provides a unified approach to equilibrium states covering both uniformly hyperbolic diffeomorphisms and subshifts of finite type \textcolor{black}{without any use of Markov partitions}.

To describe our construction, let $G_1$ be a H\"older continuous potential and consider a measure $\mu_{G_1}^u = \mu_{x,G_1}^u$ supported on a piece of unstable manifold $W_\delta^u(x)$ with the conditional Gibbs property \textcolor{black}{defined in \textsection \ref{Ces section}}. Intuitively the conditional Gibbs property  {\color{black} gives} a uniform bound on the measure of unstable Bowen balls  {\color{black} of the form}
\begin{equation*}
    \frac{\mu_{G_1}^u(B_{d_u}(y,n,\epsilon))}{e^{S_nG_1(y) - nP(G_1)}},
\end{equation*}
where $y \in W_\delta^u(x)$, $\epsilon > 0$ small, $S_nG_1(x) = \sum_{k=0}^{n-1} G_1(f^ix)$, 
{\color{black} $P(G_1)$ is the pressure and}
$B_{d_u}(y, n, \epsilon)$ denotes the Bowen ball in $W^u(x)$ with respect to the unstable metric on $W^u(x)$.

We can now {\color{black} give a brief overview of} our construction for Smale spaces. Starting from a conditional Gibbs measure for
{\color{black} a H\"older continuous function}
 $G_1$.   We {\color{black} then} define a 
  sequence of reference measures which are absolutely continuous with respect to $\mu_{G_1}^u$ and have the appropriately chosen density $e^{S_nG_2(y) - S_nG_1(y)}$ for a continuous $G_2$. Taking averaged pushforwards of the sequence of reference measures, the  weak* convergent limits are equilibrium states for $G_2$. The precise statement can be found in Theorem \ref{transform main}. 
  One way to view Theorem \ref{transform main} is as a geometric method
which transforms the Gibbs measure for $G_1$ into the Gibbs measure for $G_2$. 
\textcolor{black}{The illustrative Example \ref{ssft example} provides an explicit calculation demonstrating the transformation of the $(1/2,1/2)$-Bernoulli measure into the $(p,1-p)$-Bernoulli measure using Theorem \ref{transform main}   for the full shift on two symbols.}




{\color{black} 
Theorem \ref{transform main} can also be viewed as a new way to construct the equilibrium state for any  continuous  function $G_2$ starting  from the  equilibrium state for a  reference  H\"older potential $G_1$. }
\textcolor{black}{Moreover, Theorem \ref{transform main} extends the construction in \cite{PaPo}, where the authors study uniformly hyperbolic attractors and therefore exhibit the important property that there is an induced volume on unstable manifolds, to non-attracting uniformly hyperbolic systems.} 





The proof of Theorem \ref{transform main} relies on a growth estimate on a piece of unstable manifold which relates the pressure of two continuous potentials $G_1$ and $G_2$. {\color{black} This result is of independent interest and its} statement  can be found in Lemma \ref{pressure unstab}.  



\section{Definitions}

We now state the definition of a Smale space which is based on  \textsection 7  in  {\color{black} Ruelle's book}
 \cite{ruelle2004thermodynamic}. The definition has multiple technical conditions so we provide a couple of enlightening examples that illustrate these conditions. 

Let $X$ be a non-empty compact metric space with metric $d$. Assume there is an $\epsilon > 0$ and a map, $[\cdot , \cdot]$ with the following properties:
\begin{equation*}
[\cdot , \cdot] : \{ (x,y) \in X \times X \text{ : } d(x,y)<\epsilon\} \rightarrow X
\end{equation*}
is a continuous map such that $[x,x] = x$ and 
\begin{align*}
[[x,y] , z] & = [x,z], \tag{SS1} \\
[x,[y,z]] & =[x,z], 	\tag{SS2}		\\
f([x,y]) & = [f(x),f(y)],	\tag{SS3}
\end{align*}
when the two sides of these relations are defined.

Additionally, we require the existence of a constant $0 < \lambda < 1$ such that, for any $x \in X$ we have the following two conditions:
For $y,z \in X$ such that $d(x,y),d(x,z) < \epsilon$ and $[y,x] = x = [z,x]$, we have
\begin{equation*}
d(f(y),f(z)) \leq \lambda d(y,z); \tag{SS4}
\end{equation*}
and for $y,z \in X$ such that $d(x,y),d(x,z) < \epsilon$ and $[x,y] = x = [x,z]$, we have
\begin{equation*}
d(f^{-1}(y),f^{-1}(z)) \leq \lambda d(y,z). \tag{SS5}
\end{equation*}

\begin{definition} \label{smaledef}
Let $X$ be a compact metric space with metric $d$. Let $f : X \rightarrow X$ be a homeomorphism and $[\cdot , \cdot]$ have the properties $SS1-SS5$ above. Then we define the  Smale space to be the quadruple $(X, d, f, [\cdot , \cdot])$.
If $f : X \rightarrow X$ is also topological mixing then we call $(X, d, f, [\cdot , \cdot])$ a mixing Smale space.
\end{definition}

In essence Smale spaces are systems that exhibit a local product structure given by $[\cdot,\cdot]$ and   this product structure 
can be used to define local stable and unstable manifolds.

\begin{definition} \label{def:unstable}
For   sufficiently small $\delta > 0$ one can define the stable and unstable manifolds through $x \in X$ by 
    \begin{align*}
        W_\delta^s(x) & = \{y \in X \text{ : } y = [x,y] \text{ and } d(x,y)<\delta\}, \\
        W_\delta^u(x) & = \{y \in X \text{ : } y = [y,x] \text{ and } d(x,y)<\delta\}.
    \end{align*}
\end{definition}

From $SS4$ and $SS5$ we have that the stable and unstable manifolds are equivalently characterised in terms of the behaviour of forward and backward orbits,

$$
W_\delta^{s}(x) = \{y \in X \hbox{ : } d(f^nx, f^ny) \leq \delta, \forall n \geq 0\},
$$
and 
$$
W_\delta^{u}(x) = \{y \in X \hbox{ : } d(f^{-n}x, f^{-n}y) \leq \delta, \forall n \geq 0\}.
$$

\subsection{Examples}

The conditions $SS1 - SS5$ are perhaps best understood with illustrating examples, namely hyperbolic diffeomorphisms and subshifts of finite type.

\subsubsection{Locally maximal hyperbolic diffeomorphisms}

Let $f: M \to M$ be a $C^{1+\alpha}$ diffeomorphism on a compact Riemannian manifold and let $X \subset M$ be a closed $f$-invariant set.
\begin{definition} \label{hypdiffdef}
The map 
 $f:X \to X$ is called a (locally maximal) hyperbolic diffeomorphism if: 
\begin{enumerate}
\item 
there exists a continuous splitting $T_X M = E^s\oplus E^u$ and $C > 0$ and $0< \lambda < 1$ such that 
$$\|Df^n|E^s\| \leq C \lambda^n \hbox{ and }\|Df^{-n}|E^u\| \leq C \lambda^n$$ for $n \geq 0$; 
\item
there exists an open neighbourhood U of X such that $X=\cap_{n \in \mathbb{Z}} f^n(U)$;
\end{enumerate}
\end{definition}



The unstable manifold theory due to Hirsch and Pugh in \cite{hp} shows that uniformly hyperbolic systems are in fact Smale spaces.

\subsubsection{Subshifts of finite type} \label{SSFT smale ex}


Let $A$ be a $k \times k$ matrix with entries consisting of zeros and ones and let $A(i,j)$ denote the $(i,j)$th entry of $A$.

\begin{definition} \label{SSFTdef}
We define the one and two sided shift space $\Sigma_A^+$ and $\Sigma_A$, respectively,  by
\begin{align*}
\Sigma_A^+ & = \{\underline{x} = (x_n)_{0}^{\infty} \in \{1,\dots,k\}^{\mathbb{Z}^+} \hbox{ : } A(x_n,x_{n+1}) = 1, n \in \mathbb{Z}^+\}, \\
\Sigma_A & = \{\underline{x} = (x_n)_{-\infty}^{\infty} \in \{1,\dots,k\}^{\mathbb{Z}} \hbox{ : } A(x_n,x_{n+1}) = 1, n \in \mathbb{Z}\}.
\end{align*}
Define the two (one) sided shift map, $\sigma : \Sigma_A \rightarrow \Sigma_A$ ($\sigma : \Sigma_A^+ \rightarrow \Sigma_A^+$) by $\sigma(x_n) = x_{n+1}$. 
\end{definition}

When $A(i,j)=1$ \textcolor{black}{for all $i,j \in \{1,\dots, k\}$,} these are called  full shifts.

For $\lambda \in (0,1)$  there is a metric on $\Sigma_A$ defined by
$d(\underline{x}, \underline{y}) = \lambda^k$
where $k = \inf\{|n| \text{ : } x_n \neq y_n\}$ 
(and on $\Sigma_A^+$ there is a metric 
$d(\underline{x}, \underline{y}) = \lambda^k$
where $k = \inf\{n \text{ : } x_n \neq y_n\}$).


\begin{definition}
For each $m,n \in \mathbb{N}$, we denote by 
\begin{equation*}
[i_{-m},\dots,i_{n}] = \{\underline{x} = (x_n)_{-\infty}^{\infty} \in \Sigma_A \hbox{ : } x_{-m} = i_{-m}, \ldots, x_{n} = i_{n}\}
\end{equation*}
a  {cylinder} in $\Sigma_A$   where $i_{-m}, \cdots, i_{n} \in \{1, \cdots, k\}$ and $A(i_j, i_{j+1})=1$ for $-m \leq j \leq n-1$.
Similarly, for each $n \in \mathbb{N}$, we denote by 
\begin{equation*}
[i_0,\dots,i_{n}] = \{\underline{x} = (x_n)_{0}^{\infty} \in \Sigma_A^+ \hbox{ : } x_0 = i_0, \ldots, x_{n} = i_{n}\}
\end{equation*}
a  \emph{cylinder} in $\Sigma_A^+$ of length $n$ where $i_0, \cdots, i_{n} \in \{1, \cdots, k\}$ and $A(i_j, i_{j+1})=1$ for $0 \leq j \leq n-1$.
\end{definition}

For two sequences, $\underline{x}, \underline{y} \in \Sigma_A$ such that $x_0 = y_0$, the product map $[\cdot,\cdot]$ is given by $[\underline{x},\underline{y}] = (\dots, y_{-2}, y_{-1}, x_0, x_1, x_2, \dots)$.

For the subshift of finite type an unstable manifold through $\underline{x} \in \Sigma_A$ is simply the elements of $\Sigma_A$ which have the same past as $\underline{x}$. We will denote $\underline{x}^-$ by the sequences that have the past, $(x_n)_{-\infty}^{0}$ {\color{black} i.e., the terms agree for indices $n \leq 0$}. Stable manifolds are similarly defined with a fixed future  {\color{black} i.e., the terms agree for indices $n \geq 0$}.


\section{Constructing equilibrium states} \label{Ces section}

{\color{black}
We begin by recalling the following standard definition.

\begin{definition} \label{eqst def}
Given a continuous function $G:X \to \mathbb R$ 
$$
P(G) := \sup\left\{ h(\mu, f) + \int G d\mu \hbox{ : } \mu = \hbox{$f$-invariant probability}\right\}
$$
is the \emph{pressure} of $G$, where $ h(\mu, f)$ denotes the entropy of $\mu$.
Any measure realizing this supremum is called an \emph{equilibrium state} for $G$.
\end{definition}
\textcolor{black}{For Smale spaces every continuous potential G has} at least one equilibrium state \cite{walters}.  If $G$ is H\"older continuous then the equilibrium state is unique \cite{ruelle2004thermodynamic}.
}


We require the  following notion of a conditional Gibbs property.
\begin{definition} \label{condgibbs}
 For $y \in \textcolor{black}{W_\delta^u(x)}$, $0 < \epsilon < \delta$ and $n \in \mathbb{N}$
we define the  \emph{unstable Bowen ball} of radius $\epsilon$ by
 $$B_{d_u}(y,n,\epsilon) 
 = \{z \in W^u(x) \hbox{ : } d_u(f^iz, f^iy) < \epsilon \hbox{ for } 0 \leq i \leq n-1\}
 $$ be the Bowen ball around $y \in W_\delta^u(x)$ in the induced unstable metric $d_u$ on $W_\delta^u(x)$. 

Let $\mu^u$ be a measure supported on a piece of unstable manifold \textcolor{black}{centred at $x$}. We say that it has the conditional Gibbs property for $G$ if for every small $\epsilon > 0$ there is a constant $K = K(\epsilon) > 0$ such that, for every $y \in W_\delta^u(x)$ and $n \in \mathbb{N}$ we have,
\begin{equation*}
K^{-1} \leq \frac{\mu^u(B_{d_u}(y, n, \epsilon))}{e^{S_nG(y) - nP(G)}} \leq K.
\end{equation*}
We write $\mu^u = \mu_{G}^u$ if this conditional property holds. We may also write $\mu_{x,G}^u$ when we need to emphasis the measure is supported on a piece of unstable manifold centred at $x$.
\end{definition}

\begin{example}
    \textcolor{black}{Let $f:X \rightarrow X$ be a uniformly hyperbolic diffeomorphism. It is shown by Leplaideur \cite{leplaideur2000local} that equilibrium states for H\"older continuous potentials have a local product structure (see Definition 2.13 \cite{CPZ-2}). Therefore, equilibrium states for H\"older potentials have conditional measures on unstable manifolds that satisfy the conditional Gibbs property.}
\end{example}

\begin{example}
    \textcolor{black}{Consider the two sided subshift of finite type $\sigma_A : \Sigma_A \rightarrow \Sigma_A$. Bowen \cite{bowen-SLN} shows we can replace $G_1$ acting on $\Sigma_A$ by a homologous $G_1^{'}$ which only depends on $(x_i)_{i=0}^\infty$ without any change to the Gibbs measure $\mu_{G_1}$. We can then define a continuous function $G_1^+$ on $\Sigma_A^+$ to be equal to $G_1^{'}$. The Gibbs measure for $G_1^+$ on the one sided subshift of finite type restricted to the sequences, $\underline{y} \in \Sigma_A^+$ such that $x_0 = y_0$ and $A(x_{0}, y_1) =1$ has the conditional Gibbs property for $G_1$.}
\end{example}


\subsection{The main construction}
We are now ready to state the main construction of this section.

\begin{thm} \label{transform main}
Let $(X,d,f,[\cdot,\cdot])$ be a topologically mixing Smale space.
Let  $G_{1} : X \rightarrow \mathbb{R}$   be a 
H\"older continuous potential and let 
 $G_{2} : X \rightarrow \mathbb{R}$ be 
a continuous potential. For $\mu_{G_1}$ a.e. $x \in X$ and $\delta>0$ small, we can define a family of  measures supported on $W_\delta^u(x)$ by 
\begin{equation} \label{lambda smale}
\lambda_{n,G_2 - G_1}(A) = \frac{\int_{W_\delta^u(x)\cap A} e^{S_nG_2(y) - S_nG_1(y)} d\mu_{G_1}^u(y)}{\int_{W_\delta^u(x)} e^{S_nG_2(y) - S_nG_1(y)} d\mu_{G_1}^u(y)}, \quad n \geq 1, 
\end{equation}
where $A \subset X$ a measurable set.
Then the   
{\color{black} measures}
\begin{equation} \label{mu_n smale}
\mu_{n,G_2 - G_1} = \frac{1}{n} \sum_{i=0}^{n-1} f^i_* \lambda_{n,G_2 - G_1}, \quad n \geq 1, 
\end{equation}
supported on $f^n W_\delta^u(x)$
have  weak star accumulation  points  which are   equilibrium measures for $G_2$.
Moreover, when $G_2$ is a H\"older function then $\mu_{n,G_2 - G_1}$ converges to the unique equilibrium state 
$\mu_{G_2}$.
\end{thm}

{\color{black} 
\begin{example}
In
 the case where $f:X \rightarrow X$ is a mixing hyperbolic attractor and
  $G_1 = \varphi^{geo}$ is the geometric potential 
  then $\mu_{G_1}$ is the SRB measure, \textcolor{black}{$\mu^u_{G_1}$ is the induced volume on $W_\delta^u(x)$} and  Theorem \ref{transform main} recovers Theorem 1.2 in \cite{PaPo}. 
  \end{example}
  }

Next we will see an illuminating example illustrating Theorem \ref{transform main}. We consider the full shift on two symbols and begin with a constant potential corresponding to the $(\frac{1}{2}, \frac{1}{2})$-Bernoulli measure. Fixing $p \in (0,1)$ $( \neq 1/2)$ we show with an explicit calculation of $\mu_{n,G_2-G_1}$ that using  Theorem \ref{transform main}  we can transform the $(\frac{1}{2}, \frac{1}{2})$-Bernoulli measure into the $(p,1-p)$-Bernoulli measure. 

The $(\frac{1}{2}, \frac{1}{2})$-Bernoulli measure is a very well understood equilibrium state for the two sided subshift of finite type. Theorem \ref{transform main} can be used to explicitly calculate the measure of cylinders for the equilibrium state of any other H\"older potentials. 

\begin{example} \label{ssft example}
Let $X = \{0,1\}^{\mathbb Z}$
and let $\sigma: X \to X$ be the  full shift on two symbols given by $\sigma(x_n)_{n \in \mathbb Z} = (x_{n+1})_{n \in \mathbb Z} $.  
Let $G_1: X \to \mathbb R$ be the constant function $G_1=-\log 2$, then the associated unique equilibrium measure is the Bernoulli measure $\mu_{G_1}  = \left( \frac{1}{2}, \frac{1}{2}\right)^{\mathbb Z}$.
For $p \in (0,1)$ not equal to $1/2$, we shall consider the locally constant  potential, $G_2 : X \rightarrow \mathbb{R}$ defined at $x = (x_n)_{n=-\infty}^{+\infty}$ by
\begin{equation*}
 G_2(x) =
    \begin{cases}
      \log p & x_0 = 0\\
      \log (1-p) & x_0 = 1.
    \end{cases}       
\end{equation*}
Then the unique equilibrium measure associated to $G_2$  is the Bernoulli measure $\mu_{G_2}  = \left( p, 1-p\right)^{\mathbb Z}$.
Given any   point  $x = (x_n)_{n=-\infty}^\infty \in X$, 
$$
W_\delta^u(x) = \{ y = (y_n)_{n=-\infty}^\infty \hbox{ : }  y_i = x_i \hbox{ for } i \leq -1\}
$$
and we can identify
 $W_\delta^u(x) = \{x_{-} \} \times X^+$ where $X^+ = \{0,1\}^{\mathbb Z_+}$ and $x_{-}  = (x_n)_{n=-\infty}^{-1}$.
 The conditional measure $\mu_{G_1}^u$ on $X$ corresponds to the Bernoulli measure $\left( \frac{1}{2}, \frac{1}{2}\right)^{\mathbb Z_+}$ on $X^+$.
We can explicitly write 
\begin{align}\label{phi2-phi1}
e^{S_{n} G_2(y) - S_{n} G_1(y)} & = 
\frac{1}{2^n}p^{\#\{0 \leq i \leq n-1 \text{ : } y_i = 0\}} (1-p)^{\#\{0 \leq i \leq n-1 \text{ : } y_i = 1\}} \nonumber\\
								& = \frac{\mu_{G_2}[y_0, \dots, y_{n-1}]}{\mu_{G_1}^u[y_0, \dots, y_{n-1}]}	.
\end{align}
where we recall, $[y_0, \cdots, y_{n-1}] = \{(z_k)_{k=-\infty}^\infty \hbox{ : } z_i = y_i \hbox{ for } 0 \leq i \leq n-1\}$.
By  the definition of $\lambda_n$ we have that 
\begin{equation} \label{sigma lambda}
\sigma^{i}_* \lambda_{n}(A)
 = \frac{
 \int_{\sigma^{-i}A \cap W_\delta^u(x)}
 e^{S_{n}G_2(y)-S_{n}G_1(y)} 
  d\mu^u_{G_1}(y)
 }{
 \int_{W_\delta^u(x)} e^{S_{n}G_2(y) - S_{n}G_1(y)} d\mu^u_{G_1}(y)
 }
\end{equation}
where we have the simplifications, $P(G_1) = P(G_2)=0$ and
\begin{align*}
\int_{W_\delta^u(x)} e^{S_{n}G_2(y)-S_{n}G_1(y)}  d\mu^u_{G_1}(y)& = 
 \sum_{[y_0, \dots, y_{n-1}]}  \mu_{G_1}^u([y_0, \dots, y_{n-1}])
 \frac{\mu_{G_2}[y_0, \dots, y_{n-1}])}{\mu_{G_1}^u[y_0, \dots, y_{n-1}])} \\
 & = 1.
\end{align*}

Consider the set $A=[z_{-M}, \dots z_{-1}, z_0, z_1, \dots, z_N]$, for $M,N \in \mathbb{N}$. We will calculate $\sigma^i_*\lambda_n(A)$ for $n \in \mathbb{N}$ and $n \gg  N + M$. Notice that for $i \geq M$,
\begin{equation*}
    \sigma^{-i}(A) = \bigcup_{[y_0, \dots, y_{i-M-1}]} [y_0, \dots, y_{i-M-1}, z_{-M}, \dots, z_N].
\end{equation*}

We have that $S_nG_1$  and $S_nG_2$ are  constant on $[y_0, \dots, y_{n-1}]$ so we can rewrite the integral in equation (\ref{sigma lambda}) as a sum over the cylinders of the same length. For ease of reading, when the intersection is non-empty, let 
\begin{align*}
    \sigma^{-i}(A) \cap [y_0, \dots, y_{n-1}] & = [y_0, \dots, y_{i-M-1}, z_{-M}, \dots z_N, y_{i+N+1}, \dots, y_{n-1}], \\
    & =: \sigma^{-i}_{y_0,\dots,y_{n-1}}(A).
\end{align*}
for $M \leq i < n-N$.  We can now simplify equation (\ref{sigma lambda}) using equation (\ref{phi2-phi1}) as follows.
\begin{align*}
    \sigma^{i}_* \lambda_{n}(A) & = \sum_{\sigma^{-i}_{y_0,\dots,y_{n-1}}(A)} \mu_{G_1}^u(\sigma^{-i}_{y_0,\dots,y_{n-1}}(A)) \frac{\mu_{G_2}(\sigma^{-i}_{y_0,\dots,y_{n-1}}(A))}{\mu_{G_1}^u(\sigma^{-i}_{y_0,\dots,y_{n-1}}(A))} \\
        & = \sum_{\sigma^{-i}_{y_0,\dots,y_{n-1}}(A)} \mu_{G_2}(\sigma^{-i}_{y_0,\dots,y_{n-1}}(A)) \\
        & = \sum_{[y_0,\dots,y_{n-1}]} \mu_{G_2}(\sigma^{-i}(A) \cap [y_0, \dots, y_{n-1}]) \\
        & = \mu_{G_2}(A).
\end{align*}

Therefore,
\begin{align*}
\mu_n(A) & = \frac{1}{n} \sum_{i=0}^{n-1}  \sigma^{i}_* \lambda_{n}(A) \\
        & = \frac{1}{n} \sum_{i=0}^{M-1}  \sigma^{i}_* \lambda_{n}(A) + \frac{1}{n} \sum_{i=M}^{n-N-1}  \sigma^{i}_* \lambda_{n}(A) +  \frac{1}{n} \sum_{i=n-N}^{n-1}  \sigma^{i}_* \lambda_{n}(A) \\
        & = \frac{1}{n} \sum_{i=0}^{M-1}  \sigma^{i}_* \lambda_{n}(A) + \frac{n-(N+M)}{n} \mu_{G_2}(A) + \frac{1}{n} \sum_{i=n-N}^{n-1}  \sigma^{i}_* \lambda_{n}(A) \\
        & \xrightarrow{n\rightarrow\infty} \mu_{G_2}(A).
\end{align*}
This is consistent with Theorem \ref{transform main}, we have practised alchemy, transforming $\mu_{G_1}$ into $\mu_{G_2}$.
\end{example}

This example also hints at an interesting feature. In the construction of the SRB measure for hyperbolic attractors \cite{ruelle} 
there is no need to average the pushforwards of the induced volume on $W_\delta^u(x)$. Example \ref{ssft example} shows that even for the full shift on two symbols, there is a continuous potential such that $\sigma^n_*\lambda_n$ does not converge to the required equilibrium state. This can be seen with the following calculation. 
\begin{align*}
    \sigma^{n}_* \lambda_{n}(A) & = \sum_{[y_0,\dots,y_{n-M-1}]} \mu_{G_1}^u([y_0,\dots,y_{n-M-1}, 
    z_{-M}, \cdots, z_N
    ]) \frac{\mu_{G_2}([y_0,\dots,y_{n-M-1}
        z_{-M}, \cdots, z_N
    ])}{\mu_{G_1}^u([y_0,\dots,y_{n-M-1}
        z_{-M}, \cdots, z_N
    ])} \\
            & = \mu_{G_1}^u([z_0,\dots,z_N]) \mu_{G_2}([z_{-M},\dots,z_{-1}]) \\
            & \neq \mu_{G_2}(A).
\end{align*}

It is an interesting question to ask whether the averaging in (\ref{mu_n smale}) is required in the setting of uniformly hyperbolic attractors. Answering this would have important consequences for the rate of convergence to the equilibrium state for $G_2$.

\section{Growth of unstable manifolds for Smale spaces} \label{growthsmalesec}
The proof of Theorem \ref{transform main} relies on  {\color{black} the following} growth rate result of unstable manifolds.




\begin{lemma} \label{pressure unstab}
Let $(X,d,f,[\cdot,\cdot])$ be a mixing Smale space. Let $G_{1} : X \rightarrow \mathbb{R}$ H\"older and $G_{2} : X \rightarrow \mathbb{R}$ continuous. For a.e.($\mu_{G_1}$)  $x \in X$ and $\delta > 0$ sufficiently small, 
\begin{equation*}
P(G_2) - P(G_1) =  \lim_{n \rightarrow \infty} \frac{1}{n} \log \int_{W_\delta^u(x)} e^{S_n(G_2 - G_1)(y)} d\mu_{G_1}^u(y).
\end{equation*}
\end{lemma}

Before we prove Lemma \ref{pressure unstab}, we recall  the following simple  property.

\begin{lemma} \label{Holder Lemma}
Let $G : X \rightarrow \mathbb{R}$ be a continuous potential. For any $\tau > 0$, there is an $\epsilon > 0$ small enough such that, for any $x \in X$ and $n \in \mathbb{N}$, if $d_n(x,y) < \epsilon$ then
\begin{equation}
|S_n G(x) - S_n G(y)| \leq n \tau.
\end{equation}
\end{lemma}


\textcolor{black}{In the proof of Lemma \ref{pressure unstab} we will use Bowen's definition of the pressure (see for example \cite{walters}) using spanning and separated sets which is equivalent to the definition given in Definition \ref{eqst def} by the variational principle \cite{walters1975variational}.}

\begin{proof}[Proof of Lemma \ref{pressure unstab}]
To get an upper bound on the growth rate in Lemma \ref{pressure unstab} we proceed as follows.  
 Given $\epsilon > 0$ and $n \geq 1$, we want to create an $(n,\kappa \epsilon)$-separated  set for some $\kappa \in (0,1)$ independent of $n$ and $\epsilon$. 
{\color{black} To this end} we can choose a maximal number of points $y_i \in \textcolor{black}{f^n}W_\delta^u(x)$ $(i = 1, \cdots, N= N(n, \epsilon))$
so that $d_u(y_i, y_j) > \epsilon/2$ whenever $i \neq j$ (where $d_u$ is the induced distance on $f^nW_\delta^u(x)$). 
 By the definition of the Smale space, the map  $f^n: W^u_\delta(x) \to f^nW^u_\delta(x)$ is locally  distance expanding and thus, in particular, the points  $x_i = f^{-n}y_i$   ($i=1, \cdots, N = N(n, \epsilon)$)  form an $(n, \kappa \epsilon)$-separated set. 





Now we have constructed $\{x_i\}$, we can {\color{black} relate these points} to an integral. Let $B_{d_u}(y,n,\epsilon)$ denote the Bowen ball contained within the unstable manifold with respect to the induced metric $d_u$, then
\begin{align*}
\sum_{i =1}^N e^{S_nG_2(x_i)} & = \sum_{i =1}^N \int_{B_{d_u}(x_i,n,\epsilon)} e^{S_nG_2(x_i)} \mu_{G_1}^u(B_{d_u}(x_i,n,\epsilon))^{-1} d\mu_{G_1}^u(y), \\
						& \geq e^{-n \tau} \sum_{i =1}^N \int_{B_{d_u}(x_i,n,\epsilon)} e^{S_nG_2(y)} \mu_{G_1}^u(B_{d_u}(x_i,n,\epsilon))^{-1} d\mu_{G_1}^u(y), \\
						& \geq e^{-n\tau+nP(G_1)} K^{-1} \sum_{i =1}^N \int_{B_{d_u}(x_i,n,\epsilon)} e^{S_nG_2(y) - S_nG_1(x_i)} d\mu_{G_1}^u(y), \\
						& \geq e^{- 2n\tau + nP(G_1) } K^{-1} \sum_{i =1}^N \int_{B_{d_u}(x_i,n,\epsilon)} e^{S_nG_2(y) - S_nG_1(y)} d\mu_{G_1}^u(y), \\
						& \geq e^{- 2n\tau + nP(G_1) } K^{-1} \int_{W_\delta^u(x)} e^{S_nG_2(y) - S_nG_1(y)} d\mu_{G_1}^u(y).
\end{align*}
In particular: 
 Line 2 uses Lemma \ref{Holder Lemma} for $G_2$;
Line 3 uses the upper bound of the conditional Gibbs property of $\mu_{G_1}^u$;
Line 4 uses Lemma \ref{Holder Lemma} for $G_1$; and 
Line 5 follows from the maximality of $\{y_i\}$, in particular $W_\delta^u(x) \subset \cup_i B_{d_u}(x_i,n,\epsilon)$.
Then letting $K(n) = e^{-2n\tau} K^{-1}$ gives    
\begin{align*}
\frac{1}{n} \log Z_{1,G_2}(n, \kappa \epsilon) & \geq \frac{1}{n} \log \bigg(K(n)\int_{W_\delta^u(x)} e^{S_nG_2(y) - S_nG_1(y) + nP(G_1)} d\mu_{G_1}^u(y)\bigg).
\end{align*}
Taking a limit as $n \rightarrow \infty$ and $\epsilon \rightarrow 0$,
\begin{align*}
P(G_2) & \geq -2\tau + 	\lim_{n \rightarrow \infty} \frac{1}{n} \log \int_{W_\delta^u(x)} e^{S_nG_2(y) - S_nG_1(y)+ nP(G_1)} d\mu_{G_1}^u(y).
\end{align*}	
Since $\tau > 0$ is arbitrarily small then,
\begin{equation*}
P(G_2)  \geq  \lim_{n \rightarrow \infty} \frac{1}{n} \log \int_{W_\delta^u(x)} e^{S_nG_2(y) - S_nG_1(y) + nP(G_1)} d\mu_{G_1}^u(y).
\end{equation*}

Before starting {\color{black} on the proof of} the lower bound we {\color{black} present a simple}  result.

\begin{lemma}\label{cover0}
\textcolor{black}{F}or any $\epsilon > 0$ there exists \textcolor{black}{an} $m > 0$ such that $f^mW^u_\delta(x)$ 
is $\epsilon$-dense in $X$.  In particular, we can assume that
 $X = \cup_{y \in f^mW^u_\delta(x)} W_\epsilon^s(y).$
\end{lemma}

\begin{proof}
This is {\color{black} an immediate} consequence of the topological mixing assumption and the local product structure for Smale spaces.
\end{proof}

To get a lower bound on the growth rate in Proposition \ref{pressure unstab}, given $\epsilon > 0$ and $n \geq 1$ we want to 
construct a well chosen $(n, 2\epsilon)$-spanning set. 
We begin by choosing a {\color{black} suitable} covering of $ f^{n+m} W_\delta^u(x)$
by $\epsilon$-balls 
\begin{equation*}
B_{d_u}(y_i,  \epsilon) \hbox{ : } i=1, \cdots, N:=N(n+m, \epsilon)
\end{equation*}
contained within the unstable manifold with respect to the induced metric $d_u$ and let $A_{\epsilon} : = \{y \in f^{n+m} W_\delta^u(x) \text{ : } \nexists z \in W^u(f^{n+m}x) \text{\textbackslash} f^{n+m} W_\delta^u(x) \text{ with } d_u(y,z) < \epsilon/2\}$.
We can choose a maximal set $S= \{y_1, \cdots, y_{N(n+m,\epsilon)}\}$ with the additional property  that $d_u(y_i,y_j) > \epsilon/2$ for $i \neq j$ and $y_i \in A_\epsilon$. 
By our choice of $S$ we have that 
\begin{equation*}
A_{\epsilon} \subset \bigcup_{i=1}^{N(n+m,\epsilon)} B_{d_u}(y_i,  \epsilon/2).
\end{equation*}
By the triangle inequality we have that 
\begin{equation*}
f^{n+m} W_\delta^u(x) \subset \bigcup_{i=1}^{N(n+m,\epsilon)} B_{d_u}(y_i,  \epsilon).
\end{equation*}

\noindent
Since $B_{d_u}(f^{-(n+m)}(y_i),n+m+1, \frac{\epsilon}{4})\cap B_{d_u}(f^{-(n+m)}(y_j), n+m+1, \frac{\epsilon}{4}) = \emptyset$ for $i \neq j$, then
the disjoint union satisfies,
\begin{equation} \label{rectangle}
\bigcup_{i=1}^{N(n+m,\epsilon)} B_{d_u}(f^{-(n+m)}(y_i),n+m+1,  \epsilon/4) \subset W_\delta^u(x).
\end{equation}

We again use the property that 
$
f^n: f^{m}W^u_\delta(x) \to f^{n+m}W^u_\delta(x)
$
 locally expands distance along the unstable 
 manifold. 
In particular, this means that  the preimages $x_i := f^{-n} y_i \in  f^m(W_\delta^u(x))$ ($i=1, \cdots, N$) form an $(n, 2\epsilon)$-spanning set. [To see this we use Lemma \ref{cover0}, for  any point  $z\in \textcolor{black}{X}$  we can choose a point $y \in f^m(W^u_\delta(x))$ with $z \in W^s_\epsilon(y)$ and observe that $d(f^jz, f^jy) < \epsilon$ for $0\leq j \leq n$.] We can then choose an $x_i$ such that $\textcolor{black}{d_n(y,x_i)} < \epsilon$ since $f^n$ is locally expanding along unstable manifolds. In particular,  by the triangle inequality 
\begin{equation*}
d(f^j z, f^jx_i)  \leq d(f^j z, f^jy)+d(f^j y, f^jx_i) < 2\epsilon
\end{equation*}
 for $0\leq j \leq n-1$.

We will now use the construction of the {\color{black} points} $\{x_i\}$
{\color{black} to get the required lower bound.}
We {\color{black} first} require the following simple inequality
\begin{align*}
e^{S_{n}G_2(x_i)} & = e^{S_{n+m}G_2(f^{-m}(x_i)) - S_mG_2(f^{-m}(x_i))}\\
					& \leq e^{S_{n+m}G_2(f^{-m}(x_i)) + m ||G_2||_{\infty}}.
\end{align*}
For ease of notation, {\color{black} set} $\overline{B}(x_i) = B_{d_u}(f^{-m}(x_i),n+m+1, \frac{\epsilon}{4})$. Therefore,
\begin{align*}
\sum_{i=1}^N e^{S_nG_2(x_i)} & = \sum_{i=1}^N \int_{\overline{B}(x_i)} e^{S_nG_2(x_i)} \mu_{G_1}^u(\overline{B}(x_i))^{-1} d\mu_{G_1}^u(y) \\
						& \leq e^{m ||G_2||_{\infty}}\sum_{i=1}^N \int_{\overline{B}(x_i)} e^{S_{n+m}G_2(f^{-m}(x_i))} \mu_{G_1}^u(\overline{B}(x_i))^{-1} d\mu_{G_1}^u(y) \\
						& \leq e^{m ||G_2||_{\infty} + (n+m)\tau} \sum_{i=1}^N \int_{\overline{B}(x_i)} e^{S_{n+m}G_2(y)} \mu_{G_1}^u(\overline{B}(x_i))^{-1} d\mu_{G_1}^u(y),\\
						& \leq e^{m ||G_2||_{\infty} + (n+m)\tau + (n+m+1)P(G_1)} K \sum_{i=1}^N \int_{\overline{B}(x_i)} e^{S_{n+m}G_2(y) - S_{n+m+1}G_1(f^{-m}(x_i))} d\mu_{G_1}^u(y) \\
						& \leq e^{m ||G_2||_{\infty} + 2(n+m+1)\tau + (n+m+1)P(G_1) } K \sum_{i=1}^N \int_{\overline{B}(x_i)} e^{S_{n+m}G_2(y) - S_{n+m+1}G_1(y)} d\mu_{G_1}^u(y).
\end{align*}
Moreover, by (\ref{rectangle}) {\color{black} we can bound}
\begin{align*}
\sum_{i=1}^N \int_{\overline{B}(x_i)} e^{S_{n+m}G_2(y) - S_{n+m}G_1(y)} d\mu_{G_1}^u(y) \leq \int_{W_\delta^u(x)} e^{S_{n+m}G_2(y) - S_{n+m}G_1(y)} d\mu_{G_1}^u(y).
\end{align*}

\noindent
Letting $L(n) = e^{m ||G_2||_{\infty} + 2(n+m)\tau + P(G_1)+||G_1||_{\infty}}K$, we have
\begin{align*}
Z_{0,G_2}(n,2\epsilon) & \leq L(n) \int_{W_\delta^u(x)} e^{S_{n+m}G_2(y) - S_{n+m}G_1(y) + (n+m)P(G_1)} d\mu_{G_1}^u(y)
\end{align*}
and thus 
$$P(G_2)  \leq 2\tau + \lim_{n \rightarrow \infty} \frac{1}{n+m} \log \int_{W_\delta^u(x)} e^{S_{n+m}G_2(y) - S_{n+m}G_1(y) + (n+m)P(G_1)} d\mu_{G_1}^u(y).
$$
Again $\tau > 0$ {\color{black} can be chosen  arbitrarily small} and so
\begin{equation*}
P(G_2)  \leq  \lim_{n \rightarrow \infty} \frac{1}{n} \log  \int_{W_\delta^u(x)} e^{S_nG_2(y) - S_nG_1(y) + nP(G_1)} d\mu_{G_1}^u(y).
\end{equation*}
This concludes the proof.
\end{proof}

\section{Proof of Theorem \ref{transform main}}

In this section we will {\color{black} complete the} proof of Theorem \ref{transform main}. 
The proof 
{\color{black} follows the general  lines of}
the proof of Theorem 1.2 in \cite{PaPo}.


\begin{proof}
{\color{black} 
We begin by observing that  If we were to replace the potential $G_2$ by $G_2 + P(G_1)$ then the measures $\lambda_{n,G_2 - G_1} = \lambda_{n,G_2 - G_1 + P(G_1)}$. \textcolor{black}{Thus when we write  $\lambda_{n,G_2 - G_1}$ we are really considering $\lambda_{n,G_2 - G_1 + P(G_1)}$.}}

By Alaoglu's theorem on the weak star compactness of the space of probability measures, the measures $\mu_{n,G_2 - G_1}$ have a weak star convergent subsequence to some measure $\mu$.
Moreover, for  any continuous $F: X \to \mathbb R$ we can {\color{black} bound}
\begin{align*}
&\bigg|\int F d\mu_{n,G_2 - G_1} - \int F\circ  f  d\mu_{n,G_2 - G_1} \bigg|  \cr
&=
 \bigg|\frac{1}{n}\sum_{k=0}^{n-1} \int F\circ f^k d\lambda_{n,G_2 - G_1}
- \frac{1}{n}\sum_{k=0}^{n-1} \int F\circ f^{k+1} d\lambda_{n,G_2 - G_1}\bigg|\cr
&\leq \frac{2\|F\|_\infty}{n} \to 0 \hbox{ as } n \to +\infty\cr
\end{align*}
and, in particular,  one easily sees that $\mu$ is $f$-invariant. 

For convenience, {\color{black} we denote}
\begin{equation*}
Z_{n}^{G_2,G_1} = \int_{W_\delta^u(x)} e^{S_nG_2(y) - S_nG_1(y) +nP(G_1)} d\mu_{G_1}^u(y)
\end{equation*}
and for $A \subset X$ let,
\begin{equation*}
K_{n,A}^{G_2,G_1} = \int_{W_\delta^u(x) \cap A} e^{S_nG_2(y) - S_nG_1(y) +nP(G_1)} d\mu_{G_1}^u(y).
\end{equation*}

\begin{definition}
Given a finite partition $\mathcal P = \{P_i\}_{i=1}^N$ we say that it has size $\epsilon > 0$ if $\sup_{i}\left\{ \hbox{\rm diam}(P_i)  \right\}< \epsilon$.
\end{definition}

By Lemma \ref{Holder Lemma}, for any $\tau > 0$ there is an $\epsilon > 0$ small enough, such that if we choose a partition $\mathcal{P}$ of size $\epsilon > 0$, then for all $x,y \in A \in \bigvee_{i=0}^{n-1} f^{-i} \mathcal{P}$, we have,
\begin{equation}
|S_n G_k(x) - S_n G_k(y)| \leq n \tau
\end{equation}
for $k = 1,2$.

{\color{black} Choosing} a partition of size $\epsilon>0$,  for each element of the refined partition we can 
{\color{black} choose} an $x_A \in A \in \bigvee_{i=0}^{n-1} f^{-i} \mathcal{P}$. We now find a 
{\color{black} convenient}
 form for the integral $\int_X G_2 d\mu_{n,G_2 - G_1}$. 
 First {\color{black} we can write}
\begin{equation*}
\int_{W_\delta^u(x)} G_2(y) d\lambda_{n,G_2 - G_1}(y) = \frac{e^{nP(G_1)}}{Z_{n}^{G_2,G_1}} \int_{W_\delta^u(x)} e^{S_n(G_2 - G_1)(y)} G_2(y) d\mu_{G_1}^u(y)
\end{equation*}
{\color{black} and then}
\begin{equation*}
\int_{f^i(W_\delta^u(x))} G_2(y) df^i_*\lambda_{n,G_2 - G_1}(y) = \frac{e^{nP(G_1)}}{Z_{n}^{G_2,G_1}} \int_{W_\delta^u(x)} e^{S_n(G_2 - G_1)(y)} G_2(f^i(y)) d\mu_{G_1}^u(y).
\end{equation*}
Recalling the definition of $\mu_{n,G_2 - G_1}$ {\color{black} we can write}
\begin{align} \label{Smale integral calc}
\int_{X} G_2(y) d\mu_{n,G_2 - G_1}(y) & = \frac{e^{nP(G_1)}}{n Z_{n}^{G_2,G_1}} \int_{W_\delta^u(x)} e^{S_n(G_2 - G_1)(y)} S_nG_2(y) d\mu_{G_1}^u(y) \nonumber\\
												& = \frac{e^{nP(G_1)}}{n Z_{n}^{G_2,G_1}} \sum_{A \in \bigvee_{i=0}^{n-1} f^{-i} \mathcal{P}} \int_{W_\delta^u(x) \cap A} e^{S_n(G_2 - G_1)(y)} S_nG_2(y) d\mu_{G_1}^u(y)\nonumber \\
			&\textcolor{black}{	 \geq \frac{e^{nP(G_1)}}{n Z_{n}^{G_2,G_1}} \sum_{A \in \bigvee_{i=0}^{n-1} f^{-i} \mathcal{P}} \bigg(S_nG_2(x_A) - n \tau \bigg) \int_{W_\delta^u(x) \cap A} e^{S_n(G_2 - G_1)(y)} d\mu_{G_1}^u(y)} \nonumber \\
			& \textcolor{black}{= - \tau + \frac{1}{n} \sum_{A \in \bigvee_{i=0}^{n-1} f^{-i} \mathcal{P}} \frac{K_{n,A}^{G_2,G_1}}{Z_{n}^{G_2,G_1}} S_nG_2 (x_A).}
\end{align}

{\color{black} We next consider}
 the entropy of $\mu_{n,G_2 - G_1}$.
For $A \in \bigvee_{i=0}^{n-1} T^{-i} \mathcal{P}$, consider
\begin{align*}
\log \int_{W_\delta^u(x) \cap A} e^{S_n(G_2 - G_1)(y)} d\mu_{G_1}^u(y) & \leq \log \bigg(e^{2n \tau} \int_{W_\delta^u(x) \cap A} e^{S_n(G_2 - G_1)(x_A)} d\mu_{G_1}(y) \bigg) \\
			& = 2n\tau + S_n(G_2 - G_1)(x_A) + \log \mu_{G_1}^u(W_\delta^u(x) \cap A).  
\end{align*}
Since $\mathcal{P}$ has size $\epsilon$ then $W_\delta^u(x) \cap A \subset B_{d_u}(x_A,n, \epsilon)$. Using the conditional Gibbs property of $\mu_{G_1}^u$ we have,
\begin{align*}
\mu_{G_1}^u(W_\delta^u(x) \cap A) & \leq K e^{S_nG_1(x_A) - nP(G_1)}.
\end{align*}

\noindent
{\color{black} In particular, } this shows
\begin{align} \label{K_n bound}
\log K_{n,A}^{G_2,G_1} & \leq nP(G_1) + 2n \tau + S_n(G_2 - G_1)(x_A) + \log K + S_nG_1(x_A) - nP(G_1) \nonumber \\
							& = S_n G_2(x_A) + \log K + 2n \tau,
\end{align}
where $K > 0$ is independent of $n$ and $A$. Working from the definition of the entropy {\color{black} we can write}
\begin{align*}
H_{\lambda_{n,G_2 - G_1}} \bigg(\bigvee_{r=0}^{n-1} f^{-h}\mathcal{P} \bigg) &
 = - \sum_{A \in \bigvee_{r=0}^{n-1} f^{-h}\mathcal{P}} \lambda_{n,G_2 - G_1}(A) \log \lambda_{n,G_2 - G_1}(A) \\
			& = - \sum_{A \in \bigvee_{r=0}^{n-1} f^{-h}\mathcal P} \frac{K_{n,A}^{G_2,G_1}}{Z_{n}^{G_2,G_1}} \log \frac{K_{n,A}^{G_2,G_1}}{Z_{n}^{G_2,G_1}} \\
			& = \log Z_{n}^{G_2,G_1} - \sum_{A \in \bigvee_{r=0}^{n-1} f^{-h}\mathcal P} \frac{K_{n,A}^{G_2,G_1}}{Z_{n}^{G_2,G_1}} \log{K_{n,A}^{G_2,G_1}},
\end{align*}
where the last equality uses that, by definition $\sum_{A \in \bigvee_{r=0}^{n-1}f^{-h}\mathcal P} K_{n,A}^{G_2,G_1} =  Z_{n}^{G_2,G_1}.$

Using equation (\ref{K_n bound}) {\color{black} we have the lower bound}
\begin{align} \label{smale entropy calc}
H_{\lambda_{n,G_2 - G_1}} \bigg(\bigvee_{r=0}^{n-1} f^{-h}\mathcal{P} \bigg) & \geq Z_{n}^{G_2,G_1} - \sum_{A \in \bigvee_{r=0}^{n-1} f^{-h}\mathcal P} \frac{K_{n,A}^{G_2,G_1}}{Z_{n}^{G_2,G_1}} \bigg(S_nG_2 (x_A) + \log K + 2n\tau \bigg)\nonumber\\
			& = Z_{n}^{G_2,G_1} - \log K - 2n \tau - \sum_{A \in \bigvee_{r=0}^{n-1} f^{-h}\mathcal P} \frac{K_{n,A}^{G_2,G_1}}{Z_{n}^{G_2,G_1}} S_nG_2 (x_A).
\end{align}

Putting together (\ref{Smale integral calc}) and (\ref{smale entropy calc}),
\begin{align*}
H_{\lambda_{n,G_2 - G_1}} \bigg(\bigvee_{r=0}^{n-1} f^{-h}\mathcal{P} \bigg) & + n \int_{X} G_2(y)  d\mu_{n,G_2 - G_1}(y) \\ 
				&\geq Z_{n}^{G_2,G_1} - \log K -2n\tau - \sum_{A \in \bigvee_{r=0}^{n-1} f^{-h}\mathcal P} \frac{K_{n,A}^{G_2,G_1}}{Z_{n}^{G_2,G_1}} S_nG_2(x_A) \\  
			& \text{\hspace{10mm}} - {n\tau} + \sum_{A \in \bigvee_{i=0}^{n-1} f^{-i} \mathcal{P}} \frac{K_{n,A}^{G_2,G_1}}{Z_{n}^{G_2,G_1}} S_nG_2(x_A) \\
			& = Z_{n}^{G_2,G_1}  - \log K - 3n\tau.
\end{align*}
\noindent
We can now use this and an entropy estimate due to Misiurewicz \cite{misiurewicz} (stated in Lemma 4.5 \cite{parmenter2022constructing}) to write
$$
\begin{aligned}
q \log Z_{n}^{G_2,G_1} - q n \int_X G d \mu_{n,G_2 - G_1}- q (\log K + 3n\tau) & {\leq} q H_{\lambda_{n,G_2 - G_1}} \bigg(\bigvee_{r=0}^{n-1} f^{-h}\mathcal{P} \bigg)  \\ 
				&	 \leq n H_{\mu_{n,G_2 - G_1}} \bigg(\bigvee_{i=0}^{q-1} f^{-i}\mathcal{P} \bigg) + 2q^2 \log \text{Card}(\mathcal{P}), 
\end{aligned}
$$
which we {\color{black} can} rearrange to get,
$$
\begin{aligned}				
\frac{\log Z_{n}^{G_2,G_1}}{n} - \frac{\log K + 3n\tau}{n} - \frac{2q \log \text{Card}(\mathcal{P})}{n} & \leq \frac{H_{\mu_{n,G_2 - G_1}} \bigg(\bigvee_{i=0}^{q-1} f^{-i}\mathcal{P} \bigg)}{q} + \int_X G_2 d \mu_{n,G_2 - G_1}.
\end{aligned}
$$
Letting $n_k \to +\infty$,
\begin{align*}
P(G_2) & =  \lim_{k \rightarrow \infty} \frac{\log Z_{n_k}^{G_2,G_1}}{n_k}\cr
& \leq \lim_{k \rightarrow \infty} \bigg(\frac{H_{\mu_{n_k,G_2 - G_1}}\bigg(\bigvee_{i=0}^{q-1} f^{-i}\mathcal{P} \bigg)}{q} + \int_X G_2 d \mu_{n_k, G_2-G_1}\bigg) + 3\tau \\
			& =  \frac{H_{\mu}\bigg(\bigvee_{i=0}^{q-1} f^{-i}\mathcal{P} \bigg)}{q} + \int_X G_2 d \mu +3\tau,
\end{align*}
where we assume without loss of generality that the boundaries of the partition have zero measure.
Letting  $q \rightarrow \infty$,
\begin{equation}
 \label{final1}
P(G_2) \leq \textcolor{black}{h_{\mu}(\mathcal P)} + \int_X G_2 d \mu + 3\tau.
\end{equation}
Therefore, since $\tau$ {\color{black} can be chosen} arbitrarily and $\mu$ is an $f$-invariant probability measure, we see from the variational principle that the inequalities in equation (\ref{final1}) are actually equalities (since {$h_{\mu}(\mathcal P)  \leq  h({\mu})$}) and therefore we conclude that the measure  $\mu$ is an equilibrium state for $G_2$.
\end{proof}

\bibliographystyle{abbrv}

\bibliography{smale-alchemy-mark-rev}

\end{document}